\documentclass{amsart}[12 pt]
\usepackage{hyperref}
\usepackage{mathtools}
\usepackage{amsfonts}
\usepackage{amssymb}
\usepackage{amsthm}

\usepackage[shortlabels]{enumitem}
\usepackage{microtype}
\usepackage{fullpage}
\usepackage{mathrsfs}
\usepackage{mleftright}
\mleftright

\NewDocumentCommand{\R}{}{\mathbb{R}}

\NewDocumentCommand{\Z}{}{\mathbb{Z}}

\NewDocumentCommand{\XXd}{}{(X,d)}

\NewDocumentCommand{\MetricXXd}{o o}{\rho_{\XXd}\IfValueT{#1}{\left(#1\IfValueT{#2}{,#2}\right)}}

\NewDocumentCommand{\CzinftySpace}{o}{C_0^\infty \IfValueT{#1}{\left( #1 \right)}}
\NewDocumentCommand{\CinftySpace}{o}{C^\infty \IfValueT{#1}{\left( #1 \right)}}

\NewDocumentCommand{\Norm}{m o}{\left\| #1 \right\|\IfValueT{#2}{_{#2}}}
\NewDocumentCommand{\LtNorm}{m o}{\Norm{#1}[L^2\IfValueT{#2}{(#2)}]}

\newcommand{\clarke}{\partial_{\mathrm C}}
\newcommand{\ip}[2]{\left\langle #1,#2\right\rangle}

\begin{document}

\newtheorem{theorem}{Theorem}[section]
\newtheorem{corollary}[theorem]{Corollary}
\newtheorem{proposition}[theorem]{Proposition}
\newtheorem{lemma}[theorem]{Lemma}
\newtheorem{conjecture}[theorem]{Conjecture}
\newtheorem{problem}[theorem]{Problem}

\theoremstyle{remark}
\newtheorem{remark}[theorem]{Remark}

\theoremstyle{definition}
\newtheorem{definition}[theorem]{Definition}

\theoremstyle{remark}
\newtheorem{example}[theorem]{Example}

\theoremstyle{remark}
\newtheorem{question}[theorem]{Question}

\numberwithin{equation}{section}

\title{A Counterexample to Fourier Alignment in Single-Neuron Modular Addition}
\author{Gautam Neelakantan Memana}
\email{neelanmemana@gmail.com}
\date{\today}
\begin{abstract}
    We give a negative solution to the problem raised in \cite{levineMAISO60}. We first present a simple construction in which an initially active ReLU neuron reaches a completely inactive state in finite time and freezes at a limit whose Fourier energy is distributed equally among all nonzero real frequency classes. The counterexample holds on an open set of initial conditions, and hence on an event of positive Gaussian probability. We include an appendix by GPT 5.6 Sol that further strengthen the counterexample by showing that failure can occur for every Clarke trajectory from an open set of initial conditions, under the convention ($\operatorname{ReLU}'(0)=0$), for smooth dead-zone approximations of ReLU, and for fixed-step full-batch gradient descent. Thus single-frequency alignment is not a general consequence of training a single neuron on modular addition.
\end{abstract}
\maketitle

\section{Introduction}
Modular arithmetic has become an important test case for understanding how neural networks discover structured algorithms. In their study of grokking, Nanda, Chan, Lieberum, Smith, and Steinhardt \cite{nanda2023progress} reverse-engineered small transformers trained on modular addition and showed that the learned computation is naturally described in Fourier space: the network embeds the inputs as rotations on a circle and combines them using trigonometric identities. More recently, He, Wang, Chen, and Yang \cite{he2026mechanism} developed a theoretical account of feature learning in two-layer networks for modular addition. In the regimes they study, individual neurons tend to select single-frequency Fourier features, with frequencies competing according to their initial spectral magnitudes and phase alignments. 

These results motivate the broader representation theoretic problem formulated by Levine in \cite{maisA5_2026}. The goal is to understand which irreducible representations are selected by the training dynamics (and with what probabilities?) when a neural network is trained to perform multiplication in a finite group. For the cyclic group $\Z/p\Z$, the irreducible representations are precisely the Fourier characters, grouped into conjugate real frequency classes. The question raised in \cite{levineMAISO60} isolates an especially simple case: a single ReLU neuron trained by exact full-table cross-entropy on addition in $\Z/p\Z$, without weight decay. It asks whether, conditional on being initially active, the normalized weights must converge to a direction supported on a single nonzero real Fourier class (see \eqref{intro::nonzero real fourier class}). 
The main purpose of the present article is to show that this conclusion fails using a simple counterexample presented in Section \ref{short counterexample}. Then, in the Appendix \ref{appendix} we establish several strengthened counterexamples showing that the failure persists under a number of natural continuous and discrete training dynamics. 

The conclusion is not that Fourier structure is absent, but that single-frequency, or monosemantic, Fourier structure is not forced by the training dynamics. The two mechanisms exhibited here are familiar ones: a unit may die, or it may remain active while memorizing a single table entry. The persistent branch realizes, in the smallest possible model, the
memorization side of the memorization-to-generalization dichotomy associated with grokking. Both mechanisms use an exact dead region of the activation, so the corresponding question remains open for everywhere-positive activations such as softplus. We also refer the reader to \cite{maisO92_2026} which is a quantitative successor to \cite{levineMAISO60}.

The permanently inactive behavior belongs to the classical dying-ReLU
phenomenon \cite{lu2020dying}. A complementary literature studies the implicit bias and directional convergence of gradient descent on separable data and in homogeneous models \cite{soudry2018implicit,lyu2020gradient,ji2020directional}. Those general results do not apply to the trajectories constructed here, because the training loss does not tend to zero: on the persistent trajectory it converges to $\frac{p^2-1}{p^2}\log p$, while on the dead trajectories it freezes at a positive value. For modular addition itself, Zhong et al. showed
that the same task can support qualitatively different Clock and Pizza mechanisms \cite{zhong2023clock}. Morwani et al. derived Fourier-feature
emergence from margin maximization in a stylized model
\cite{morwani2024feature}; their modular-addition theorem assumes width
$m\geq4(p-1)$, so the case $m=1$ studied here lies outside that margin
mechanism. Finally, the value assigned to $\operatorname{ReLU}'(0)$ is an
automatic-differentiation convention whose numerical consequences have been
studied explicitly \cite{bertoin2021numerical}.

\subsection{The problem}
Let $C_p:=\mathbb{Z}/p\mathbb{Z}$, where $p\geq 5$ is prime. A
single-neuron network has parameters
\[
u,v,w\in\mathbb{R}^{C_p}.
\]
For an input $(a,b)\in C_p\times C_p$, define its gate, or preactivation, by
\[
z_{a,b}:=u(a)+v(b),
\]
and its activation by
\[
h_{a,b}:=\operatorname{ReLU}(z_{a,b}),
\qquad
\operatorname{ReLU}(s):=\max{s,0}.
\]
We call the $(a,b)$ gate active, inactive, or at the kink according as
$z_{a,b}$ is positive, negative, or zero. Thus, the term ``the $(a,b)$
gate'' refers to the scalar preactivation $z_{a,b}=u(a)+v(b)$. For the
input $(a,b)$ and output class $c\in C_p$, the corresponding logit is
$
h_{a,b}w(c).
$

The full-table cross-entropy loss is
\begin{align}
\mathcal{L}(u,v,w)
&=
\frac{1}{p^2}
\sum_{a,b\in C_p}
\left[
\log\left(
\sum_{c\in C_p}e^{h_{ab}w(c)}
\right)
-
h_{ab}w(a+b)
\right].
\label{intro::loss function}
\end{align}
There is no weight decay. For a real-valued function $x:C_p\to\mathbb{R}$, we write
\begin{align}\label{intro::renormalization}
x_0
=
x-\frac{1}{p}
\left(\sum_{a\in C_p}x(a)\right)\mathbf{1}
\end{align}
for the recentered version of \(x\). For every
\(1\leq\zeta\leq (p-1)/2\), let \(\Pi_\zeta\) denote the orthogonal
projection onto the real Fourier space
\begin{align}
E_\zeta
:=
\operatorname{span}_{\mathbb{R}}
\left\{
a\mapsto
\cos\left(\frac{2\pi\zeta a}{p}\right),
\quad
a\mapsto
\sin\left(\frac{2\pi\zeta a}{p}\right)
\right\}.
\label{intro::real Fourier class}
\end{align}
This space corresponds to the conjugate pair of complex characters
\begin{align}\label{intro::nonzero real fourier class}
    [\rho_\zeta]=\{\rho_\zeta,\rho_{-\zeta}\},\qquad\rho_\zeta(a):=e^{2\pi i\zeta a/p},
\end{align}
and will be called a nonzero real Fourier class. 

Whenever
$
\|u_0\|^2+\|v_0\|^2+\|w_0\|^2>0,
$
the combined Fourier-energy fraction of the class
\([\rho_\zeta]\) is
\begin{align}\label{intro::combined fraction}
F_\zeta(u,v,w) =\frac{\|\Pi_\zeta u_0\|^2+\|\Pi_\zeta v_0\|^2+\|\Pi_\zeta w_0\|^2}{
\|u_0\|^2+\|v_0\|^2+\|w_0\|^2}.
\end{align}
\begin{definition}\label{intro:: purity}
   For \(\delta>0\), we say that the parameter triple
\((u,v,w)\) is \((\delta,[\rho_\zeta])\)-pure if
\[
F_\zeta(u,v,w)\geq 1-\delta.
\] 
\end{definition}

\begin{remark}
    Definition \ref{intro:: purity} essentially says that all but at most a $\delta$-fraction of centered Fourier energy is contained in the single real Fourier class \([\rho_\zeta]\).
\end{remark}
The ReLU function is differentiable away from the origin, with
\[
\operatorname{ReLU}'(s)
=
\begin{cases}
0, & s<0,\\
1, & s>0,
\end{cases}
\]
but it is not differentiable at \(s=0\). The point \(s=0\) is called
the \emph{ReLU kink}. Its Clarke generalized derivative is
\[
\partial_C\operatorname{ReLU}(0)=[0,1].
\]
We use the Clarke generalized-gradient framework and its standard calculus
rules \cite[Chapter~2]{clarke1990}. Accordingly, the continuous-time training
dynamics are interpreted as the Clarke differential inclusion
\begin{align}\label{intro::clarke inclusion}
\dot{\theta}(t)\in-\partial_C\mathcal{L}(\theta(t)),
\qquad
\theta(t):=(u(t),v(t),w(t)).
\end{align}
A \textit{measurable assignment of complete Clarke trajectories} means a Borel-measurable choice of one complete Clarke trajectory for each initial condition. We say that the neuron is \emph{active at initialization} if
\begin{align*}
u(a)+v(b)>0
\end{align*}
for at least one pair \((a,b)\in C_p\times C_p\).

\begin{problem}[\cite{levineMAISO60}]\label{intro:: the problem}
Condition on the event that the neuron is active at initialization.
Prove or refute the following statement: almost surely on this event,
the normalized parameters
\begin{align*}
\frac{(u(t),v(t),w(t))}
{\|(u(t),v(t),w(t))\|}
\end{align*}
converge as \(t\to\infty\), and there exists a nonzero frequency
\(\zeta\) such that the limiting direction is
\((\delta,[\rho_\zeta])\)-pure for every \(\delta>0\).
\end{problem}

Since the same frequency \(\zeta\) is required to work for every
\(\delta>0\), the purity condition is equivalent to
\[
F_\zeta=1.
\]
In other words, the entire centered Fourier energy of the limiting
direction must lie in a single nonzero real Fourier class. For further
background on this problem and its representation-theoretic
motivation, we refer the reader to \cite{maisA5_2026}.
\begin{remark}
    One can easily show that a single neuron cannot assign different active inputs to the different values of $(a+b\pmod p)$, and hence cannot achieve perfect accuracy on the modular-addition table. The Problem \ref{intro:: the problem} therefore is not whether this neuron can learn the complete modular-addition algorithm, but whether the training dynamics nevertheless select a single Fourier frequency in the normalized parameter direction.
\end{remark}
Now, we state the main theorem. 
\begin{theorem}\label{introduction::main theorem}
    For every prime $p\ge 5$, there is a nonempty set $U\subset \R^{3p}$ consisting of initially active parameters, and a measurable assignment of complete Clarke trajectories such that every selected trajectory starts in $U$ and reaches a stationary dead-neuron state in finite time and its normalized limiting direction is not supported on any one nonzero real Fourier class. Moreover, 
    \begin{align*}
    \mathbb{P}(\theta(0)\in U|\ \text{initially active})>0.
\end{align*}
Consequently, the almost-sure assertion in Problem \ref{intro:: the problem} fails under the measurable Clarke-trajectory convention. 
\end{theorem}
\begin{remark}
    For $p=2$, the centered real function space is one-dimensional and equals the
unique nonzero Fourier component. For $p=3$, it is two-dimensional and equals
the single conjugate frequency class. Thus single-frequency purity is
automatic for $p=2,3$, independently of the dynamics. The first genuinely
multifrequency case is $p=5$.
\end{remark}
\begin{remark}
    It seems likely that a theorem similar to Theorem \ref{introduction::main theorem} can be proven for the \textit{Purity} problem raised in \cite{L+26b} for the symmetric group $S_3$ using ideas from Section~\ref{short counterexample}, but we do not pursue it in this article. 
\end{remark}

\textbf{AI tool disclosure and acknowledgements.} In keeping with the principles of transparency, attribution, and human responsibility articulated in the \href{https://leidendeclaration.ai}{Leiden Declaration on Artificial Intelligence and Mathematics}, the author discloses the following uses of artificial-intelligence tools. All results in this article, except those presented in Appendix~\ref{appendix}, were worked out and written by the author, with helpful input from conversations with GPT-5.6 Sol, including assistance with the literature review and with understanding and clarifying the problem. The strengthened results in Appendix~\ref{appendix} were developed through conversations with GPT-5.6 Sol, which also produced the initial draft of the appendix; the author subsequently reviewed and edited the text. 

The author also thanks Lionel Levine for the encouragement to post this article as part of the MAIS project.
\bibliographystyle{amsalpha}
\section{Proof of \ref{introduction::main theorem}}
\begin{proof}
Fix numbers $\varepsilon, A,B>0$ such that 
\begin{align*}
    B>\varepsilon/2
, \quad A^2 > \frac{p}{2(p-1)}\varepsilon^2.\end{align*}
Consider the initial condition 
\begin{align}\label{proof::initial condition}
    &u(0)=v(0)=\varepsilon/2, \quad u(a)=v(a)=-B \ (a\neq 0), \quad w(0)=-A, \quad w(c)=0\ (c\neq 0).
\end{align}
It is easy to verify that exactly one of the table entry is active as $u(0)+v(0)=\varepsilon>0$, whereas 
\begin{align*}
    u(0)+v(b) = u(a)+v(0)=\frac{\varepsilon}{2}-B<0,
\end{align*}
for $a, b\neq 0$, and $u(a)+v(b)=-2B<0$ when $a, b\neq 0$. 

As long as the $(0,0)$ gate remains positive, all other gates contribute zero gradient. By symmetry, the wrong output coordinates remain equal. Set 
\begin{align*}
    s=u(0)+v(0), \quad a=w(0), \quad b=w(j) \ (j\neq 0), \quad g=b-a. 
\end{align*}
Initially $s(0)=\varepsilon$ and $g(0)=A$. For the sole active example, whose label is $0$, its softmax probability of the correct class is
\begin{align*}
    \pi_0:=\frac{e^{sa}}{e^{sa}+(p-1)e^{sb}} =\frac{1}{1+(p-1)e^{sg}}.
\end{align*}
We can use $\pi_0$ to simplify the ODEs in the gradient descent algorithm corresponding to the loss function \eqref{intro::loss function}. Using basic chain rule we get 
\begin{align}
    \dot{s}&=-\frac{2}{p^2} (1-\pi_0)g, \label{proof::ode s}\\
    \dot{a}&=\frac{s}{p^2}(1-\pi_0),\label{proof::ode a}\\ 
    \dot{b}&=-\frac{s}{p^2}\frac{1-\pi_0}{p-1},\\ 
    \dot{g}&=-\frac{s}{p(p-1)}(1-\pi_0).  \label{proof::ode g}
\end{align}
Now, we make the following important observation. 
\begin{align}\label{proof::equation satisfied by g and s}
    g(t)^2-\frac{p}{2(p-1)}s(t)^2=A^2-\frac{p}{2(p-1)}\varepsilon^2. 
\end{align}
This is not difficult to see. Whenever $s>0, g>0$, the division of \eqref{proof::ode g} by \eqref{proof::ode s} yields
\begin{align*}
    \frac{dg}{ds}=\frac{ps}{2(p-1)g}.
\end{align*}
Thus,
\begin{align*}
    \frac{d}{ds}\left(g^2-\frac{p}{2(p-1)}s^2\right)=0,
\end{align*}
which, along with the initial values, gives us \eqref{proof::equation satisfied by g and s}. 
\\

Put 
\begin{align*}
    \Gamma=A^2-\frac{p}{2(p-1)}\varepsilon^2>0. 
\end{align*}
\eqref{proof::equation satisfied by g and s} shows that $g(t)\ge \sqrt{\Gamma}$ while $s\geq 0$. Moreover, $s(t)g(t)\geq 0$, and so softmax $\pi_0\leq \frac{1}{p}$ and therefore
\begin{align*}
    1-\pi_0\geq \frac{p-1}{p}.
\end{align*}
Consequently,
\begin{align*}
    \dot{s}\leq -\frac{2(p-1)}{p^3}\sqrt{\Gamma}<0. 
\end{align*}
Hence, $s$ reaches $0$ at a finite time $T$, with explicit upper bound 
\begin{align*}
    T\leq \frac{p^3\varepsilon}{2(p-1)\sqrt{\Gamma}}.
\end{align*}
But, observe that $g(T)=\sqrt{\Gamma}>0$. Moreover, the hit is transverse, i.e., 
\begin{align}\label{proof::tranversely}
    \lim_{t\to T^{-}} \dot{s}(t)= -\frac{2(p-1)}{p^3}\sqrt{\Gamma}<0.
\end{align}
We can thus conclude that no other gate can turn on before $T$. Indeed $u(0)$ and $v(0)$ decrease at the same rate, all $u(a), v(b)$ with nonzero index remain fixed, and every initially inactive pre-activation is therefore constant or decreasing. Since $u(0)-v(0)$ is conserved and initially zero, 
\begin{align}\label{proof::at T}
    u(0,T)=v(0,T)
=0.\end{align}
At time $T$, choose the Clarke derivative of the $\textrm{ReLU}$ at zero to be $0$. Every pre-activation is then non-positive and every active value is zero. The output derivatives vanish because they contain a factor $h_{ab}$, and the input -weight derivatives vanish by the chosen Clarke slope. Therefore
\begin{align*}
    (u(t), v(t),w(t))=(u(T), v(T), w(T))\quad (t\ge T),
\end{align*}
hence we have a complete Clarke trajectory. 
\subsection{Fourier spectrum of the frozen state}
Using \eqref{proof::at T} 
\begin{align*}
    u(T)=v(T)=(0,-B, \cdots, -B),
\end{align*}
and hence the centered normalizations are
\begin{align*}
    u(T)_0=v(T)_0=B \left[e_0-\frac{1}{p}\mathbf{1}\right],
\end{align*}
where $e_0$ is the point mass at $0$. Since $g(T)=b(T)-a(T)>0$, we get
\begin{align*}
    w(T)_0=-g(T)\left[e_0-\frac{1}{p}\mathbf{1}\right].
\end{align*}
Thus, all three centered limiting vectors are nonzero multiples of the same centered point mass $q:=\left[e_0-\frac{1}{p}\mathbf{1}\right]$. Now, taking the unitary Fourier transform gives 
\begin{align*}
    \hat{q}(0)=0, \quad \hat{q}(\zeta)=\frac{1}{\sqrt{p}} \ (\zeta\neq 0). 
\end{align*}
Therefore, $\|q\|^2=(p-1)/p$ and the real class $[\rho_{\zeta}]=\{\zeta, -\zeta\}$ contains $2/p$ energy. Therefore any nonzero real Fourier class captures exactly 
\begin{align*}
    \frac{2/p}{(p-1)/p}=\frac{2}{p-1}
\end{align*}
of the energy of $q$. Since the three centered vectors $u(T)_0, v(T)_0, w(T)_0$ are scalar multiples of $q$, their combined fraction \eqref{intro::combined fraction} is
\begin{align*}
    F_{\zeta}(u(T), v(T), w(T))=\frac{2}{p-1}\quad \forall \ \zeta\neq 0. 
\end{align*}
For $p\ge 5$, this number is strictly less than $1$. In particular, for any 
\begin{align*}
    0<\delta< 1-\frac{2}{p-1},
\end{align*}
the limiting direction is not $(\delta, [\rho_{\zeta}])$-pure for any frequency class. 
\subsection{One trajectory to positive probability}
A single perfectly symmetric initialization has Gaussian probability zero, so one final robustness is required. Let $\theta_*$ denote the initial condition \eqref{proof::initial condition}. Up to the hitting time $T$, its trajectory lies in the open smooth chamber
\begin{align*}
    u(0)+v(0), \quad u(a)+v(b)<0\quad ((a,b)\neq (0,0)).
\end{align*}
On this chamber the vector field is smooth. The trajectory reaches the boundary $u(0)+v(0)$ transversely by \eqref{proof::tranversely}, while all other boundary inequalities regain negative margins. Standard continuous dependence for smooth ODEs, together with the implicit-function theorem for a transverse hitting time, therefore gives a neighborhood $U_0$ of $\theta_*$ such that every initial condition in $U_0$:
\begin{itemize}
    \item begins with exactly one active gate;
    \item follows the same smooth chamber until a unique nearby transverse hitting time;
    \item reaches the hitting time before any other gate becomes active; and 
    \item has a hitting endpoint continuously close to $\theta(T)$.
\end{itemize}
For each of these trajectories, choose ReLU slope at $0$ hitting gate and keep the trajectory constant afterward. At the base endpoint, the maximum of the finitely many Fourier fractions is $\frac{2}{p-1}<1$. The fractions are continuous wherever combined centered norm is nonzero. After shrinking this neighborhood if necessary, set $U:U_0$ on which every frozen endpoint still satisfies $\max_{\zeta}F_{\zeta}<1$. Thus no selected limiting direction from $U$ is supported on one frequency class. 

Finally, independent standard Gaussians have a strictly positive probability density on all of $\R^{3p}$ and hence every open set has a positive probability. Since $U$ consists entirely of initially active parameters, 
\begin{align*}
    \mathbb{P}(\theta(0)\in U|\ \text{initially active})>0.
\end{align*}
On this positive probability event the normalized weights converge, but the limit is not pure. The selected trajectories on $U$ depend measurably on the initial condition. They may therefore be patched with any measurable assignment outside $U$, producing a global measurable assignment for which the conclusion of Problem \ref{intro:: the problem} fails on an event of positive conditional probability. Alternatively, Theorem \ref{thm:app-dead} removes the selection issue entirely, since every Clarke trajectory from its open set of initial conditions fails single-frequency alignment. This proves Theorem \ref{introduction::main theorem}.
\end{proof}\label{short counterexample}
\appendix 
\section{Strengthened counterexamples}\label{appendix}

\noindent\textit{This appendix is adapted from a note drafted by GPT-5.6 Sol in conversation with the author (see \cite{gpt56sol2026strengthened}). The author subsequently reviewed and edited the statements and proofs.}

This appendix records stronger versions of the preceding counterexample. The
first construction remains persistently active and is independent of the
choice of Clarke selection. The remaining results show robustness under the
zero-at-the-kink convention, smooth dead-zone activations, and fixed-step
full-batch gradient descent.

\subsection{Strengthened statements}

\begin{theorem}[Persistent, selection-independent Clarke counterexample]
\label{thm:app-persistent}
Let $p\geq5$ be a prime. There is a nonempty open set
$U_{\mathrm{slide}}\subset\R^{3p}$ such that, for every initial condition
$\theta_0\in U_{\mathrm{slide}}$, the Clarke differential inclusion
\eqref{intro::clarke inclusion} has a unique complete trajectory. This trajectory
satisfies:
\begin{enumerate}[label=\textup{(\roman*)}]
\item exactly one input, after relabeling $(0,0)$, has strictly positive ReLU
activation, while all cross inputs $(0,j)$ and $(j,0)$ with $j\neq0$
eventually lie on ReLU kink faces;
\item $u(0,t)+v(0,t)\to+\infty$ and $\Norm{\theta(t)}\to\infty$;
\item $\theta(t)/\Norm{\theta(t)}$ converges;
\item for every nonzero real Fourier class,
\[
 \lim_{t\to\infty}F_\zeta(u(t),v(t),w(t))=\frac2{p-1}.
\]
\end{enumerate}
In particular, every Clarke trajectory from every point of
$U_{\mathrm{slide}}$ fails single-frequency alignment. The set has positive
probability under independent standard Gaussian initialization, even after
conditioning on initial activity.
\end{theorem}

\begin{theorem}[Every Clarke trajectory and $\textrm{ReLU}'(0)=0$]
\label{thm:app-dead}
Let $p\geq5$ be a prime. There is a nonempty open set
$U_{\mathrm{dead}}\subset\R^{3p}$ of initially active states such that:
\begin{enumerate}[label=\textup{(\roman*)}]
\item every Clarke trajectory from $U_{\mathrm{dead}}$ reaches a completely
inactive state in finite time and is constant thereafter;
\item the deterministic continuous-time dynamics obtained by assigning
$\textrm{ReLU}'(0)=0$ have the same behavior;
\item after shrinking $U_{\mathrm{dead}}$ if necessary, there exists
$\delta_p>0$ such that every terminal state satisfies
\[
 \max_{\zeta\neq0}F_\zeta(\theta_\infty)\leq1-\delta_p.
\]
\end{enumerate}
\end{theorem}

\begin{theorem}[$C^\infty$ dead-zone smoothing]
\label{thm:app-smoothing}
Fix $\epsilon>0$. Let $\psi\in C^\infty(\R;[0,1])$ be nondecreasing and
satisfy
\[
 \psi(t)=0\ (t\leq0),
 \qquad
 \psi(t)>0\ (t>0),
 \qquad
 \psi(t)=1\ (t\geq1).
\]
Set
\begin{equation}
 \sigma_\epsilon(z):=z\psi(z/\epsilon).
\label{app:sigma}
\end{equation}
Replace ReLU by $\sigma_\epsilon$ in the network. Then there is a nonempty
open set $U_{\epsilon,\mathrm{sm}}$ of initially active states such that the
unique smooth gradient-flow trajectory converges to a finite nonzero limit
$\theta_\infty$ and
\[
 \max_{\zeta\neq0}F_\zeta(\theta_\infty)<1.
\]
At a symmetric base point,
\[
 F_\zeta(\theta_\infty)=\frac2{p-1}
 \qquad(\zeta\neq0).
\]
\end{theorem}

\begin{theorem}[Fixed-step full-batch gradient descent]
\label{thm:app-discrete}
Let $p\geq5$ be a prime and fix $\eta>0$. Consider vanilla full-batch gradient descent
\[
 \theta_{n+1}=\theta_n-\eta\nabla\mathcal L(\theta_n)
\]
with the convention $\textrm{ReLU}'(0)=0$. There is a nonempty open set
$U_{\eta,\mathrm{GD}}$ of initially active states such that every gate is
strictly inactive after the first update. Consequently
$\theta_n=\theta_1$ for all $n\geq1$, the normalized iterates converge, and
their limit is not single-frequency. At a symmetric base point,
\[
 F_\zeta(\theta_1)=\frac2{p-1}
 \qquad(\zeta\neq0).
\]
\end{theorem}

\begin{remark}[Scope]
Theorem~\ref{thm:app-smoothing} concerns smooth activations with an exact dead
half-line. It does not cover softplus or arbitrary everywhere-positive
smoothings. Theorem~\ref{thm:app-discrete} concerns vanilla full-batch gradient
descent with a fixed step size and does not by itself cover momentum, Adam,
stochastic minibatches, or other stateful optimizers.
\end{remark}

\subsection{One-example derivatives and Clarke calculus}

For an input whose correct label is $y\in C_p$, let $h\geq0$ be its scalar
activation and define
\begin{equation}
 \ell_y(h,w)=\log\sum_{c\in C_p}e^{hw(c)}-hw(y).
\label{app:ell}
\end{equation}
Set
\begin{equation}
 q_c(h,w)=\frac{e^{hw(c)}}{\sum_{d\in C_p}e^{hw(d)}}
\label{app:q}
\end{equation}
and
\begin{equation}
 g_y(h,w)=\partial_h\ell_y(h,w)
 =\sum_cq_c(h,w)w(c)-w(y).
\label{app:g}
\end{equation}
When $z=u(a)+v(b)>0$, this example contributes
\begin{align}
 \dot u(a)&=-\frac1{p^2}g_y(z,w),
 &\label{app:udot}\\
 \dot v(b)&=-\frac1{p^2}g_y(z,w),
 &\label{app:vdot}\\
 \dot w(c)&=-\frac{z}{p^2}\bigl(q_c(z,w)-\mathbf1_{\{c=y\}}\bigr).
 &\label{app:wdot}
\end{align}
At $z<0$, the example contributes nothing. Let
\begin{equation}
 \mu(w):=\frac1p\sum_cw(c),
 \qquad
 \gamma_y(w):=\mu(w)-w(y)=g_y(0,w).
\label{app:gamma}
\end{equation}

\begin{lemma}[Exact positive-slope Clarke kink]
\label{lem:app-kink}
Suppose $\gamma_y(w_*)>0$. For
\[
 H_y(z,w):=\ell_y(\textrm{ReLU}(z),w),
\]
one has
\begin{equation}
 \clarke H_y(0,w_*)
 =\{(\alpha\gamma_y(w_*),0):0\leq\alpha\leq1\}.
\label{app:kink-subdiff}
\end{equation}
The function $H_y$ is Clarke regular at $(0,w_*)$. After composing $z$ with
the affine gate $u(a)+v(b)$, the example may therefore contribute
\[
 \alpha\gamma_y(w_*)\bigl(e_a^{(u)}+e_b^{(v)}\bigr),
 \qquad0\leq\alpha\leq1,
\]
to the $(u,v)$-gradient and contributes zero to the $w$-gradient. For a
finite sum of example losses, suppose additionally that every other example is
either strictly active, strictly inactive, or lies at a zero-coefficient kink.
Then the coefficients of the simultaneous positive-slope kink examples may be
selected independently.
\end{lemma}

\begin{proof}
At differentiability points with $z<0$, $H_y(z,w)=\log p$, so the gradient is
zero. At differentiability points with $z>0$,
\[
 \partial_zH_y(z,w)=g_y(z,w)\longrightarrow\gamma_y(w_*)
\]
as $(z,w)\to(0,w_*)$, while
\[
 \nabla_wH_y(z,w)=z\bigl(q(z,w)-e_y\bigr)\longrightarrow0.
\]
Thus the limiting gradients are $(0,0)$ and $(\gamma_y(w_*),0)$, and their
closed convex hull is \eqref{app:kink-subdiff}. The directional derivative is
\[
 H_y'(0,w_*;\dot z,\dot w)
 =\gamma_y(w_*)\max\{\dot z,0\},
\]
which is the support function of that segment
\cite[Proposition~2.1.2]{clarke1990}. Hence the ordinary directional derivative
agrees with the Clarke generalized directional derivative, and $H_y$ is
Clarke regular in the sense of \cite[Definition~2.3.4]{clarke1990}.

At a zero-coefficient kink, the same calculation gives the singleton Clarke
subdifferential $\{0\}$ and zero directional derivative, so that summand is
also Clarke regular. Under the hypothesis in the final sentence of the lemma,
every other summand is therefore smooth or Clarke regular. The exact
finite-sum rule for regular functions and the affine chain rule
\cite[Proposition~2.3.3 and Theorem~2.3.9]{clarke1990} give equality with the
corresponding Minkowski sum. This permits the positive-slope kink coefficients
to be selected independently.
\end{proof}

\begin{lemma}[Local semiconvexity at a positive-slope kink]
\label{lem:app-semiconvex}
If $\gamma_y(w_*)>0$, then $H_y(z,w)=\ell_y(\textrm{ReLU}(z),w)$ is semiconvex in a
neighborhood of $(0,w_*)$.
\end{lemma}

\begin{proof}
The function has two smooth branches,
\[
 H_y^-(z,w)=\log p\quad(z<0),
 \qquad
 H_y^+(z,w)=\ell_y(z,w)\quad(z>0).
\]
They agree continuously on $z=0$. Their tangential $w$-gradients agree there
and are both zero, while
\[
 \partial_zH_y^-(0,w)=0,
 \qquad
 \partial_zH_y^+(0,w)=\gamma_y(w).
\]
After shrinking the neighborhood, $\gamma_y(w)\geq m>0$, and the Hessians of
both smooth branches are bounded below by $-CI$.

Restrict $H_y$ to an arbitrary affine line. Away from a crossing of $z=0$,
its second derivative is bounded below by $-C$ times the squared line speed.
At a crossing, its first derivative has a nonnegative jump because the normal
derivative jumps from $0$ to $\gamma_y(w)>0$. Adding a sufficiently large
quadratic makes every such line restriction convex. Hence $H_y$ is locally
semiconvex.
\end{proof}

\begin{lemma}[Du--Hu--Lee balance invariant along Clarke trajectories]
\label{lem:app-balance}
Along every Clarke trajectory,
\begin{equation}
 \Norm{u(t)}^2+\Norm{v(t)}^2-\Norm{w(t)}^2
 =
 \Norm{u(0)}^2+\Norm{v(0)}^2-\Norm{w(0)}^2.
\label{app:balance}
\end{equation}
\end{lemma}

\begin{proof}
For ordinary gradient flow, this is the one-hidden-unit specialization of
the nodewise balance invariant of Du, Hu, and Lee
\cite[Theorem~2.1]{du2018algorithmic}: the incoming weight block is $(u,v)$
and the outgoing weight block is $w$. The additional point needed here is
that the same identity holds for every Clarke selection.

In the defining Clarke gradient-limit representation, the approximating
points may be chosen outside the finite union of kink hyperplanes
\[
 \{u(a)+v(b)=0\},\qquad (a,b)\in C_p^2,
\]
which is a null set. Thus, at every such differentiability point, a single
example with preactivation $z=u(a)+v(b)$ satisfies
\[
 \ip{u}{\nabla_u\mathcal L_{a,b}}
 +\ip{v}{\nabla_v\mathcal L_{a,b}}
 =
 \ip{w}{\nabla_w\mathcal L_{a,b}}.
\]
If $z>0$, both sides equal $z\,g_y(z,w)$, up to the common factor $p^{-2}$;
if $z<0$, both sides vanish. Summing over the examples gives the identity
for the ordinary gradient at the approximating points.

The identity is linear in the gradient and therefore passes to limits and
convex combinations. Consequently, for every
\[
 \xi=(\xi_u,\xi_v,\xi_w)\in\clarke\mathcal L(u,v,w),
\]
one has
\[
 \ip{u}{\xi_u}+\ip{v}{\xi_v}=\ip{w}{\xi_w}.
\]
Along a Clarke trajectory, choose
$\xi(t)\in\clarke\mathcal L(\theta(t))$ such that
$\dot\theta(t)=-\xi(t)$. Then, for almost every $t$,
\[
 \frac{\partial}{\partial t}
 \left(\Norm{u}^2+\Norm{v}^2-\Norm{w}^2\right)
 =
 -2\left(
 \ip{u}{\xi_u}+\ip{v}{\xi_v}-\ip{w}{\xi_w}
 \right)
 =0.
\]
This proves \eqref{app:balance}.
\end{proof}

\subsection{The persistent sliding trajectory}

Distinguish $0\in C_p$ and write
\begin{equation}
 x:=u(0)+v(0).
\label{app:x}
\end{equation}
For $j\neq0$, define the cross-gate preactivations
\begin{equation}
 r_j:=u(0)+v(j),
 \qquad
 s_j:=u(j)+v(0),
\label{app:cross}
\end{equation}
and set
\begin{equation}
 \gamma_j:=\mu-w(j).
\label{app:gammaj}
\end{equation}
Fix $\gamma_*>0$ and choose
\begin{equation}
 X>\frac{p-1}{e\gamma_*}.
\label{app:X}
\end{equation}
Let $U_{\mathrm{slide}}$ consist of initial conditions satisfying
\begin{align}
 x&>X,
 &\label{app:U1}\\
 u(a)+v(b)&<0\quad\text{for every }(a,b)\neq(0,0),
 &\label{app:U2}\\
 \gamma_j&>\gamma_*\quad(j\neq0).
 &\label{app:U3}
\end{align}
This is a nonempty open set. For example, take
\[
 u(0)=v(0)=X,
 \qquad
 u(j)=v(j)=-3X\quad(j\neq0),
\]
and
\[
 w(0)=2p\gamma_*,
 \qquad
 w(j)=0\quad(j\neq0),
\]
then perturb slightly. Condition \eqref{app:U3} implies
\[
 w(0)>\mu>w(j)\quad(j\neq0),
\]
because
\[
 w(0)-\mu=\sum_{j\neq0}(\mu-w(j)).
\]
Set
\begin{equation}
 d_j:=w(0)-w(j)>0.
\label{app:dj}
\end{equation}

While only $(0,0)$ is strictly active, define
\begin{equation}
 A=A(x,w):=w(0)-\sum_cq_c(x,w)w(c)
 =\sum_{j\neq0}q_j(x,w)d_j>0.
\label{app:A}
\end{equation}
The central example gives
\begin{align}
 \dot w(0)&=\frac{x}{p^2}(1-q_0),
 &\label{app:w0dot}\\
 \dot w(j)&=-\frac{x}{p^2}q_j\quad(j\neq0).
 &\label{app:wjdot}
\end{align}
Consequently,
\begin{align}
 \dot\mu&=0,
 &\label{app:mudot}\\
 \dot\gamma_j&=\frac{x}{p^2}q_j>0,
 &\label{app:gammadot}\\
 \dot d_j&=\frac{x}{p^2}(1-q_0+q_j)>0.
 &\label{app:ddot}
\end{align}
Moreover,
\begin{equation}
 A=\frac{\sum_{j\neq0}d_je^{-xd_j}}{1+\sum_{j\neq0}e^{-xd_j}}
 \leq\sum_{j\neq0}d_je^{-xd_j}
 \leq\frac{p-1}{ex},
\label{app:A-bound}
\end{equation}
since $ze^{-z}\leq e^{-1}$ for $z\geq0$.

\subsubsection{Sequential capture}

At any time, let
\[
 R:=\{j\neq0:r_j=0\},
 \qquad
 C:=\{j\neq0:s_j=0\},
\]
and put $k=|R|$, $\ell=|C|$. For $j\in R$, choose the row-kink coefficient
\begin{equation}
 \alpha_j:=\frac{A}{(k+1)\gamma_j},
\label{app:alpha}
\end{equation}
and for $j\in C$, choose
\begin{equation}
 \beta_j:=\frac{A}{(\ell+1)\gamma_j}.
\label{app:beta}
\end{equation}

The construction is understood inductively over the strata determined by the
captured sets $R$ and $C$. At the entry time of a stratum, the central gate is
strictly positive, the captured cross gates are zero, every uncaptured cross
gate and every off-cross gate is strictly negative, and the strict bounds
$x>X$ and $\gamma_j>\gamma_*$ hold. On that stratum, define the selected vector
field using \eqref{app:alpha}--\eqref{app:beta} and solve the resulting smooth
ODE up to the first possible exit. The calculations in
Lemmas~\ref{lem:app-sliding} and~\ref{lem:app-no-offcross}, made conditionally
up to that exit time, show that $x$ and every $\gamma_j$ increase, the kink
coefficients remain admissible, captured gates stay at zero, and no off-cross
gate can reach zero. Hence the only possible exit is an impact at which one or
more uncaptured cross gates reach zero. Add every such gate to $R$ or $C$ and
restart the smooth ODE on the new stratum. Continuity preserves all remaining
strict inequalities across the impact. Since each impact captures at least one
of the $2(p-1)$ cross gates, there can be at most $2(p-1)$ impacts; Lemma~\ref{lem:app-finite-capture} below shows that no uncaptured cross gate can
persist forever. This induction removes any circular dependence among the sign
pattern, the sliding equations, and the invariants used below.

\begin{lemma}[Admissibility and sliding]
\label{lem:app-sliding}
The coefficients \eqref{app:alpha}--\eqref{app:beta} lie in $[0,1]$ and give
\begin{align}
 \dot u(0)&=\frac{A}{(k+1)p^2},
 &\dot v(j)&=-\frac{A}{(k+1)p^2}\quad(j\in R),
 &\label{app:row-slide}\\
 \dot v(0)&=\frac{A}{(\ell+1)p^2},
 &\dot u(j)&=-\frac{A}{(\ell+1)p^2}\quad(j\in C).
 &\label{app:col-slide}
\end{align}
Hence $\dot r_j=0$ for $j\in R$ and $\dot s_j=0$ for $j\in C$.
\end{lemma}

\begin{proof}
By \eqref{app:A-bound}, $x\geq X$, and \eqref{app:gammadot},
\[
 \alpha_j\leq\frac{p-1}{(k+1)eX\gamma_*}<1,
\]
and similarly for $\beta_j$. Lemma~\ref{lem:app-kink} permits these
coefficients. The central example contributes $A/p^2$ to $\dot u(0)$, while
each of the $k$ captured row gates contributes
$-\alpha_j\gamma_j/p^2=-A/((k+1)p^2)$. Thus
\[
 \dot u(0)=\frac1{p^2}\left(A-k\frac{A}{k+1}\right)
 =\frac{A}{(k+1)p^2}.
\]
A captured $v(j)$ receives only its own row-gate contribution. The column
calculation is identical.
\end{proof}

In particular,
\begin{equation}
 \dot x=\frac{A}{p^2}\left(\frac1{k+1}+\frac1{\ell+1}\right)>0.
\label{app:xdot}
\end{equation}
An uncaptured row gate has $v(j)$ constant and hence
$\dot r_j=\dot u(0)>0$; uncaptured columns behave similarly.

\begin{lemma}[No off-cross activation]
\label{lem:app-no-offcross}
Throughout the construction, $(0,0)$ is the only strictly active input.
Captured cross inputs have activation zero, and every other input has strictly
negative preactivation.
\end{lemma}

\begin{proof}
If neither coordinate is captured, the corresponding entries of $u$ and $v$
remain constant. If $j\in R$, then $v(j)=-u(0)$. For an uncaptured column
index $i$, $u(i)+v(0)<0$, and hence
\[
 u(i)+v(j)=u(i)-u(0)<-v(0)-u(0)=-x<0.
\]
The other mixed case is identical. If both indices are captured, the
preactivation is exactly $-x<0$.
\end{proof}

\begin{lemma}[Finite capture]
\label{lem:app-finite-capture}
Every cross gate is captured after finitely many impacts, and every individual
impact occurs at a finite time.
\end{lemma}

\begin{proof}
Captured gates never leave zero, and there are only $2(p-1)$ cross gates.
Suppose that after some time $T$ the captured sets are constant and at least
one row gate remains uncaptured. If
\[
 \int_T^\infty A(t)\ d t<\infty,
\]
then \eqref{app:row-slide}--\eqref{app:col-slide} imply that all coordinates of
$u$ and $v$ remain bounded. Captured coordinates are affine negatives of
$u(0)$ or $v(0)$, and uncaptured coordinates are constant.
Lemma~\ref{lem:app-balance} then bounds $w$.

The gaps $d_j$ are bounded below by their positive values at time $T$, and
$x\geq X$. The state remains in a compact set on which the continuous
function $A(x,w)=\sum_{j\neq0}q_j(x,w)d_j$ is strictly positive. Thus
$A\geq a_0>0$, contradicting integrability. Therefore
\[
 \int_T^\infty A(t)\ d t=\infty.
\]
By \eqref{app:row-slide}, $u(0,t)\to\infty$. An uncaptured row gate equals
$r_j(t)=u(0,t)+v(j,T)$ and must hit zero at a finite time, a contradiction.
The column case is identical.

At a simultaneous impact, add every gate that reaches zero to $R$ or $C$.
There are finitely many impacts. On each stratum the selected vector field is
smooth, so the construction gives an absolutely continuous Clarke trajectory.
It is complete because \eqref{app:A-bound} and \eqref{app:xdot} bound the
$u$- and $v$-velocities on finite intervals, while
$|\dot w(c)|\leq x/p^2$ and $x$ remains bounded on finite intervals.
\end{proof}

\subsubsection{The full sliding face}

After the capture time, write
\[
 a=u(0),\qquad b=v(0),\qquad x=a+b.
\]
The cross-face relations force
\begin{equation}
 u=(a,-b,\ldots,-b),
 \qquad
 v=(b,-a,\ldots,-a).
\label{app:full-face}
\end{equation}
Every off-cross preactivation is $-x<0$. Since $k=\ell=p-1$,
Lemma~\ref{lem:app-sliding} gives
\begin{equation}
 \dot a=\dot b=\frac{A}{p^3},
 \qquad
 \dot x=\frac{2A}{p^3}.
\label{app:full-flow}
\end{equation}
In particular, $a-b$ is constant.

\begin{proposition}[Divergence and one-example memorization]
\label{prop:app-divergence}
Along the full-face trajectory,
\[
 x(t)\to\infty,
 \qquad
 \Norm{w(t)}\to\infty,
 \qquad
 \Norm{\theta(t)}\to\infty.
\]
The loss of the central example tends to zero, every other example has loss
$\log p$, and
\begin{equation}
 \mathcal L(t)\longrightarrow\frac{p^2-1}{p^2}\log p.
\label{app:loss-limit}
\end{equation}
\end{proposition}

\begin{proof}
If $x$ were bounded, then \eqref{app:full-face} and constancy of $a-b$ would
bound $u$ and $v$. Lemma~\ref{lem:app-balance} would bound $w$. The increasing
gaps $d_j$ would remain positive, placing the trajectory in a compact set on
which $A$ has a positive minimum. Equation \eqref{app:full-flow} would then
force $x$ to grow at a uniform positive rate, a contradiction. Hence
$x\to\infty$. The balance invariant forces $\Norm w\to\infty$ and therefore
$\Norm\theta\to\infty$.

The central loss is
\[
 \ell_0(x,w)=\log\left(1+\sum_{j\neq0}e^{-xd_j}\right)\longrightarrow0,
\]
because every $d_j$ is bounded below by a positive constant. Every other
activation is zero, so its loss is $\log p$. This proves
\eqref{app:loss-limit}.
\end{proof}

\subsection{Projective asymptotics and the flat Fourier spectrum}

Let
\begin{equation}
 w_-:=\frac1{p-1}\sum_{j\neq0}w(j),
 \qquad
 d:=w(0)-w_-,
\label{app:wminus}
\end{equation}
and write
\begin{equation}
 e_j:=w(j)-w_-,
 \qquad
 \sum_{j\neq0}e_j=0.
\label{app:ej}
\end{equation}

\begin{lemma}[Bounded wrong-class spread]
\label{lem:app-spread}
The diameter
\[
 \max_{j\neq0}w(j)-\min_{j\neq0}w(j)
\]
is nonincreasing. Consequently all $e_j(t)$ remain uniformly bounded.
\end{lemma}

\begin{proof}
For $j,k\neq0$,
\[
 \frac\partial{\partial t}(w(j)-w(k))=-\frac{x}{p^2}(q_j-q_k).
\]
The sign of $q_j-q_k$ is the sign of $w(j)-w(k)$. Hence a maximal wrong
weight decreases at least as fast as a minimal wrong weight, so the upper Dini
derivative of the diameter is nonpositive.
\end{proof}

Since $\mu$ is constant,
\begin{equation}
 w(0)=\mu+\frac{p-1}{p}d,
 \qquad
 w(j)=\mu-\frac1p d+e_j.
\label{app:wdecomp}
\end{equation}
Thus
\begin{equation}
 \Norm w^2=p\mu^2+\frac{p-1}{p}d^2+\sum_{j\neq0}e_j^2.
\label{app:wnorm}
\end{equation}
If $\Delta=a-b$, then
\begin{equation}
 \Norm u^2+\Norm v^2
 =p(a^2+b^2)=\frac p2x^2+\frac p2\Delta^2.
\label{app:uvnorm}
\end{equation}
Combining \eqref{app:balance}, \eqref{app:wnorm}, \eqref{app:uvnorm}, and
Lemma~\ref{lem:app-spread} gives
\begin{equation}
 \frac p2x^2-\frac{p-1}{p}d^2=O(1).
\label{app:balance-asymptotic}
\end{equation}
Since $d>0$ and $x\to\infty$,
\begin{equation}
 \frac d x\longrightarrow
 \kappa_p:=p\sqrt{\frac1{2(p-1)}}.
\label{app:kappa}
\end{equation}
Define
\begin{equation}
 q:=e_0-\frac1p\mathbf{1}
 =\left(\frac{p-1}{p},-\frac1p,\ldots,-\frac1p\right).
\label{app:qpoint}
\end{equation}
From \eqref{app:full-face},
\begin{equation}
 u_0=xq,
 \qquad
 v_0=xq,
\label{app:uvcentered}
\end{equation}
where $u_0$ and $v_0$ are as in \eqref{intro::renormalization}.
Equations \eqref{app:ej}--\eqref{app:wdecomp} give
\begin{equation}
 w_0=dq+e,
\label{app:wcentered}
\end{equation}
where $e(0)=0$ and $e(j)=e_j$ for $j\neq0$. Hence
\begin{equation}
 \frac{u_0}{x}\to q,
 \qquad
 \frac{v_0}{x}\to q,
 \qquad
 \frac{w_0}{x}\to\kappa_pq.
\label{app:centered-limit}
\end{equation}
For the uncentered limit, $a/x\to1/2$ and $b/x\to1/2$. Let
\[
 V:=\left(\frac12,-\frac12,\ldots,-\frac12\right).
\]
Then
\[
 \frac ux\to V,
 \qquad
 \frac vx\to V,
 \qquad
 \frac wx\to\kappa_pq.
\]
Since
\[
 \Norm V^2=\frac p4,
 \qquad
 \Norm q^2=\frac{p-1}{p},
 \qquad
 \kappa_p^2\Norm q^2=\frac p2,
\]
one has $\Norm\theta/x\to\sqrt p$, and therefore
\begin{equation}
 \frac{\theta(t)}{\Norm{\theta(t)}}
 \longrightarrow
 \frac1{\sqrt p}(V,V,\kappa_pq).
\label{app:projective-limit}
\end{equation}
Finally,
\[
 \widehat q(0)=0,
 \qquad
 \widehat q(\zeta)=\frac1{\sqrt p}\quad(\zeta\neq0),
\]
so
\begin{equation}
 \frac{\Norm{\Pi_\zeta q}^2}{\Norm q^2}=\frac2{p-1}.
\label{app:qfraction}
\end{equation}
By \eqref{app:centered-limit}, the same limit holds for the combined fraction
$F_\zeta$.

\subsection{Why every Clarke trajectory is the sliding trajectory}

The abstract uniqueness mechanism used in this subsection is classical.
For a semiconvex function, its subdifferential is hypomonotone
\cite[Example~12.28(b)]{rockafellarwets1998}, and the resulting comparison
estimate yields uniqueness of subgradient trajectories by Gronwall's
inequality; compare
\cite[Lemma~2.1(a)]{marcellinthibault2006}. The loss-specific content here is
the verification that the present ReLU cross-entropy loss is locally
semiconvex in a neighborhood of every point of the constructed sliding
trajectory.
\begin{lemma}[Hypomonotonicity]
\label{lem:app-hypomonotone}
Suppose a locally Lipschitz function $F$ is $C$-semiconvex on a convex
neighborhood $V$, meaning
\[
 x\longmapsto F(x)+\frac C2\Norm x^2
\]
is convex on $V$. Then for all $x_1,x_2\in V$ and
$\xi_i\in\clarke F(x_i)$,
\begin{equation}
 \ip{\xi_1-\xi_2}{x_1-x_2}\geq-C\Norm{x_1-x_2}^2.
\label{app:hypomonotonicity}
\end{equation}
\end{lemma}

\begin{proof}
This is the standard hypomonotonicity estimate for the subdifferential of a
semiconvex function; see
\cite[Example~12.28(b)]{rockafellarwets1998}. For completeness, adding
$C\Norm{\cdot}^2/2$ converts the Clarke subdifferential locally into the
convex subdifferential. Thus
\[
 \xi_i+Cx_i\in\partial\left(F+\frac C2\Norm{\cdot}^2\right)(x_i).
\]
Monotonicity of the convex subdifferential gives
\[
 \ip{(\xi_1+Cx_1)-(\xi_2+Cx_2)}{x_1-x_2}\geq0,
\]
which is \eqref{app:hypomonotonicity}.
\end{proof}

\begin{lemma}[Gronwall uniqueness for semiconvex subgradient flow]
\label{lem:app-local-uniqueness}
If $F$ is $C$-semiconvex on $V$, then two absolutely continuous solutions of
\[
 \dot x(t)\in-\clarke F(x(t))
\]
that start at the same point and remain in $V$ are identical.
\end{lemma}

\begin{proof}
This is the standard Gronwall comparison argument for hypomonotone
subgradient flows; compare
\cite[Lemma~2.1(a)]{marcellinthibault2006}. Choose
$\xi_i(t)\in\clarke F(x_i(t))$ with $\dot x_i=-\xi_i$. By
Lemma~\ref{lem:app-hypomonotone}, almost everywhere,
\[
 \frac12\frac\partial{\partial t}\Norm{x_1-x_2}^2
 =-\ip{\xi_1-\xi_2}{x_1-x_2}
 \leq C\Norm{x_1-x_2}^2.
\]
Gronwall's inequality and equal initial data give $x_1=x_2$.
\end{proof}

\begin{proposition}[Uniqueness along the constructed path]
\label{prop:app-uniqueness}
For every $\theta_0\in U_{\mathrm{slide}}$, every Clarke trajectory starting
from $\theta_0$ coincides with the constructed sliding trajectory.
\end{proposition}

\begin{proof}
Let $\bar\theta$ be the constructed trajectory. At every point of
$\bar\theta$, the central gate is strictly positive, every off-cross gate is
strictly negative, and every zero gate is a cross gate with target $j\neq0$.
Moreover, \eqref{app:gammadot} gives
\[
 \gamma_j(t)\geq\gamma_j(0)>\gamma_*>0.
\]
Lemma~\ref{lem:app-semiconvex} shows that every zero-gate term is locally
semiconvex in its gate and output variables. We use here that semiconvexity is
preserved under affine precomposition. Indeed, if $f$ is $C$-semiconvex and
$x\mapsto Ax+b$ is affine, then
\[
 f(Ax+b)+\frac{C\|A\|_{\mathrm{op}}^2}{2}\Norm{x}^2
 =
 \left(f+\frac C2\Norm{\cdot}^2\right)(Ax+b)
 +\frac C2\left(\|A\|_{\mathrm{op}}^2\Norm{x}^2-\Norm{Ax+b}^2\right).
\]
The first term on the right is convex, while the Hessian of the quadratic
part of the second is
\[
 C\left(\|A\|_{\mathrm{op}}^2I-A^{\mathsf T}A\right)\succeq0.
\]
Thus $f\circ(A\,\cdot+b)$ is $C\|A\|_{\mathrm{op}}^2$-semiconvex. Applying
this to the affine gate map shows that every zero-gate loss term is locally
semiconvex as a function of the full parameter vector. The smooth central term
is locally semiconvex after increasing the constant, and the inactive terms
are locally constant. Hence $\mathcal L$ is locally semiconvex near every
point of $\bar\theta$.

If another Clarke trajectory $\theta$ with the same initial condition first
separated from $\bar\theta$ at time $T$, choose a convex neighborhood of
$\theta(T)=\bar\theta(T)$ on which $\mathcal L$ is semiconvex. Both paths
remain in this neighborhood for a short interval after $T$, and
Lemma~\ref{lem:app-local-uniqueness} forces them to agree there, a
contradiction. Thus they agree for all time.
\end{proof}

Theorem~\ref{thm:app-persistent} follows from the construction,
Proposition~\ref{prop:app-divergence}, the projective limit
\eqref{app:projective-limit}, the Fourier calculation \eqref{app:qfraction},
and Proposition~\ref{prop:app-uniqueness}.

\subsection{A finite-time dead-neuron region}

Begin with a symmetric base point
\begin{align}
 u(0)&=v(0)=\frac{s_0}{2},
 &u(j)&=v(j)=-M\quad(j\neq0),
 &\label{app:dead-u}\\
 w(0)&=-B,
 &w(j)&=0\quad(j\neq0),
 &\label{app:dead-w}
\end{align}
where $M,B>0$ and $s_0>0$ is small. Only $(0,0)$ is active. Define
\begin{equation}
 \gamma:=\mu-w(0)>0
\label{app:dead-g}
\end{equation}
Here $u(0)$, $v(0)$, and $w(0)$ denote coordinates indexed by $0$. To
distinguish these from evaluation at the initial time, set
\begin{equation}
 \gamma_{\mathrm{in}}:=\gamma\big|_{t=0}.
\label{app:gin}
\end{equation}
Also define
\begin{equation}
 D(s,w):=\sum_cq_c(s,w)w(c)-w(0).
\label{app:D}
\end{equation}
The log-partition function
\[
 \Phi(s):=\log\left(\frac1p\sum_ce^{sw(c)}\right)
\]
is convex, so
\begin{equation}
 D(s,w)=\Phi'(s)-w(0)\geq\Phi'(0)-w(0)=g.
\label{app:D-lower}
\end{equation}
While the central gate is positive,
\begin{align}
 \dot s&=-\frac2{p^2}D(s,w),
 &\label{app:dead-sdot}\\
 \dot \gamma&=-\frac{s}{p^2}(1-q_0).
 &\label{app:dead-gdot}
\end{align}
Let $[0,\tau)$ be the maximal interval on which
\[
 s(t)>0
 \qquad\text{and}\qquad
 \gamma(t)>0.
\]
On this interval, \eqref{app:D-lower} gives $D(s,w)\geq \gamma>0$, so
$\dot s<0$ and $s$ may be used as the independent variable. Moreover,
\[
 \frac{\partial \gamma}{\partial s}
 =\frac{s(1-q_0)}{2D(s,w)}
 \leq\frac{s}{2\gamma}.
\]
Thus
\begin{equation}
 \frac{\partial(\gamma^2)}{\partial s}\leq s.
\label{app:g2ds}
\end{equation}
Integrating from $s(t)$ to $s_0$ gives
\begin{equation}
 \gamma(t)^2
 \geq \gamma_{\mathrm{in}}^2-\frac{s_0^2-s(t)^2}{2}
 \geq \gamma_{\mathrm{in}}^2-\frac{s_0^2}{2}.
\label{app:g-lower}
\end{equation}
Choose $s_0$ so that
\begin{equation}
 \gamma_{\mathrm{in}}^2>\frac{s_0^2}{2},
\label{app:dead-condition}
\end{equation}
and set
\[
 \gamma_*:=\sqrt{\gamma_{\mathrm{in}}^2-\frac{s_0^2}{2}}>0.
\]
Then $\gamma(t)\geq \gamma_*$ throughout $[0,\tau)$, so the maximal interval cannot end
because $\gamma$ vanishes. Furthermore,
\[
 \dot s=-\frac2{p^2}D(s,w)\leq-\frac{2\gamma_*}{p^2}.
\]
Hence $s$ reaches zero in finite time. By maximality, $\tau=T$ is precisely
this first hitting time. Both $u(0)$ and $v(0)$ decrease on $[0,T]$, so every
other preactivation remains strictly negative up to and including time $T$.

\begin{lemma}[No continuation through the dead face]
\label{lem:app-no-continuation}
At time $T$, every Clarke trajectory remains at the same parameter vector for
all later times.
\end{lemma}

\begin{proof}
At time $T$, every noncentral preactivation is strictly negative. Let
$\tau_+\in(T,\infty]$ be the first later time at which one of these
preactivations reaches zero, with $\tau_+=\infty$ if no such time exists. On
$[T,\tau_+)$, all noncentral gates remain strictly negative, so only the central
gate can contribute. At $s=0$, Lemma~\ref{lem:app-kink} gives the scalar
inclusion
\begin{equation}
 \dot s\in\left[-\frac{2g(T)}{p^2},0\right].
\label{app:scalar-inclusion}
\end{equation}
For $s<0$, $\dot s=0$, while for $s>0$,
\eqref{app:dead-sdot} gives $\dot s<0$.

Let $s^+=\max\{s,0\}$ and $s^-=\max\{-s,0\}$. The chain rule for absolutely
continuous functions gives $\dot s^+\leq0$ almost everywhere, so
$s^+(t)=0$ after $T$. Also, an absolutely continuous function has derivative
zero almost everywhere on each level set; therefore $\dot s=0$ almost
everywhere on $\{s=0\}$ and on $\{s<0\}$. Hence $\dot s^-=0$ almost
everywhere and $s^-(t)=0$. Thus $s(t)=0$ throughout $[T,\tau_+)$.

The activation is therefore zero on $[T,\tau_+)$, so the output gradient
vanishes. The realized kink coefficient is zero almost everywhere, and all
parameters are constant on this interval. If $\tau_+<\infty$, every noncentral
preactivation would consequently retain its strictly negative value from time
$T$ at time $\tau_+$, contradicting the definition of $\tau_+$. Hence
$\tau_+=\infty$, and the parameter vector is constant for every $t\geq T$.
Under the deterministic convention $\textrm{ReLU}'(0)=0$, the vector field is
explicitly zero at time $T$.
\end{proof}

At the symmetric base point, every parameter block retains a two-level form.
Any centered two-level vector is a scalar multiple of
$q=e_0-p^{-1}\mathbf{1}$, so
\begin{equation}
 F_\zeta(\theta(T))=\frac2{p-1}
 \qquad(\zeta\neq0).
\label{app:dead-fraction}
\end{equation}
The hitting is transverse because $\dot s(T^-)\leq-2g_*/p^2<0$. Continuous
dependence of the smooth pre-hitting ODE and transversality imply that the
hitting time and terminal vector depend continuously on initial data near the
base point. The strict negativity of all other gates and positivity of $g(T)$
persist. Since there are finitely many frequency classes and the denominator
of $F_\zeta$ is nonzero at the base point, a sufficiently small open
neighborhood satisfies
\[
 \max_{\zeta\neq0}F_\zeta(\theta_\infty)<1.
\]
Together with Lemma~\ref{lem:app-no-continuation}, this proves
Theorem~\ref{thm:app-dead}.

\subsection{Smooth dead-zone activations}

Let $\sigma_\epsilon$ be defined by \eqref{app:sigma}. It is $C^\infty$,
vanishes on $(-\infty,0]$, agrees with ReLU on $[\epsilon,\infty)$, and for
$s>0$ satisfies
\begin{equation}
 \sigma_\epsilon'(s)
 =\psi(s/\epsilon)+\frac{s}{\epsilon}\psi'(s/\epsilon)>0,
 \qquad
 \frac{\sigma_\epsilon(s)}{\sigma_\epsilon'(s)}\leq s.
\label{app:sigma-properties}
\end{equation}
Use initial data of the form \eqref{app:dead-u}--\eqref{app:dead-w}, so only
the central example has positive preactivation. Write
\[
 s=u(0)+v(0),
 \qquad
 h=\sigma_\epsilon(s),
 \qquad
 g=\mu-w(0).
\]
The central-only region is invariant: $u(0)$ and $v(0)$ decrease by the same
amount, all cross gates become more negative, and all off-cross gates remain
unchanged. The dynamics are
\begin{align}
 \dot s&=-\frac2{p^2}\sigma_\epsilon'(s)D(h,w),
 &\label{app:smooth-sdot}\\
 \dot g&=-\frac{\sigma_\epsilon(s)}{p^2}(1-q_0).
 &\label{app:smooth-gdot}
\end{align}
Convexity of the log-partition function in $h$ gives
\begin{equation}
 D(h,w)\geq g.
\label{app:smooth-D}
\end{equation}
Using $s$ as the independent variable,
\[
 \frac{\partial g}{\partial s}
 =\frac{\sigma_\epsilon(s)(1-q_0)}
 {2\sigma_\epsilon'(s)D(h,w)}
 \leq\frac{\sigma_\epsilon(s)}
 {2g\sigma_\epsilon'(s)}.
\]
Therefore
\begin{equation}
 \frac{\partial(g^2)}{\partial s}
 \leq\frac{\sigma_\epsilon(s)}{\sigma_\epsilon'(s)}.
\label{app:smooth-g2}
\end{equation}
Define
\begin{equation}
 J(s_0):=\int_0^{s_0}
 \frac{\sigma_\epsilon(r)}{\sigma_\epsilon'(r)}\ d r.
\label{app:J}
\end{equation}
By \eqref{app:sigma-properties}, $J(s_0)\leq s_0^2/2$. If
\begin{equation}
 g(0)^2>J(s_0),
\label{app:smooth-condition}
\end{equation}
then
\begin{equation}
 g(t)^2\geq g(0)^2-J(s_0)=:g_*^2>0.
\label{app:smooth-glower}
\end{equation}
Thus $s(t)$ decreases. Its limit must be zero; otherwise
\eqref{app:smooth-sdot} and \eqref{app:smooth-D} would give
$\dot s\leq-c<0$ eventually. Moreover,
\begin{equation}
 \int_0^\infty\sigma_\epsilon(s(t))\ dt
 \leq\frac{p^2}{2g_*}
 \int_0^{s_0}\frac{\sigma_\epsilon(r)}{\sigma_\epsilon'(r)}\ dr
 <\infty.
\label{app:smooth-integrable}
\end{equation}
The output velocities are bounded by a constant times
$\sigma_\epsilon(s(t))$, so $w(t)$ converges. Since $u(0)-v(0)$ is constant
and $u(0)+v(0)=s(t)\to0$, both $u$ and $v$ converge. Hence
$\theta(t)\to\theta_\infty$.

At the symmetric base point, two-level symmetry is preserved, so every
centered parameter block is a scalar multiple of $q$ and
\[
 F_\zeta(\theta_\infty)=\frac2{p-1}.
\]
The terminal map is continuous near the base point. Indeed, fix a small level
$\delta>0$. The hitting time of $s=\delta$ is transverse, and the smooth ODE
depends continuously on initial data up to that time. The remaining tail is
uniformly small because
\[
 \int_{\{s\leq\delta\}}\sigma_\epsilon(s(t))\ d t
 \leq\frac{p^2}{2g_*}
 \int_0^\delta\frac{\sigma_\epsilon(r)}{\sigma_\epsilon'(r)}\ d r
 \longrightarrow0
\]
as $\delta\downarrow0$, while the remaining changes in $u(0)$ and $v(0)$
are at most $\delta/2$. Shrinking to a small open neighborhood preserves a
strict multifrequency gap. This proves Theorem~\ref{thm:app-smoothing}.

\subsection{One-step death under full-batch gradient descent}

Fix $\eta>0$ and use the convention $\textrm{ReLU}'(0)=0$. Start from a symmetric
base point of the form \eqref{app:dead-u}--\eqref{app:dead-w}, with only the
central gate active. Let
\[
 s_0=u_0(0)+v_0(0)>0,
 \qquad
 g_0=\mu_0-w_0(0)>0.
\]
The first full-batch update gives
\begin{equation}
 s_1=s_0-\frac{2\eta}{p^2}D(s_0,w_0).
\label{app:discrete-update}
\end{equation}
By \eqref{app:D-lower}, $D(s_0,w_0)\geq g_0$. Choose
\begin{equation}
 0<s_0<\frac{2\eta g_0}{p^2}.
\label{app:discrete-condition}
\end{equation}
Then $s_1<0$. Each row-cross or column-cross gate decreases during the first
update, while every off-cross gate is unchanged. Hence all gates are strictly
negative at step $1$. The gradient vanishes at every later step, so
\[
 \theta_n=\theta_1\qquad(n\geq1).
\]
At the symmetric base point, the update preserves the two-level form of each
parameter block, and therefore
\[
 F_\zeta(\theta_1)=\frac2{p-1}
 \qquad(\zeta\neq0).
\]
All gate inequalities are strict, and the one-step update is continuous on a
neighborhood containing no initial kink. A sufficiently small open
neighborhood therefore has the same one-step death and a strict
multifrequency terminal gap. This proves Theorem~\ref{thm:app-discrete}.
\bibliography{bibliography}

\end{document}